\documentclass[]{aspm1}

\articleinfo{}{}{}

\usepackage{verbatim}
\usepackage{amssymb}
\usepackage{amsbsy}
\usepackage{amscd}
\usepackage{amsmath}
\usepackage{amsthm}
\usepackage[mathscr]{eucal}

\theoremstyle{plain}
\newtheorem{thm}{Theorem}[section]
\newtheorem{lemma}[thm]{Lemma}
\newtheorem{cor}[thm]{Corollary}

\theoremstyle{definition}

\newtheorem{remark}[thm]{Remark}
\newtheorem{remarks}[thm]{Remarks}

\theoremstyle{remark}
\newtheorem{assumption}{Assumption}

\newcommand\mstrut{{\phantom{.}}}
\newcommand\newdot{{\kern.8pt\cdot\kern.8pt}}
\newcommand\bull{{\hbox{\bf .}}}

\font\sevenrm=cmr7

\newcommand\E{\mathbb{E}}
\newcommand\R{\mathbb{R}}

\newcommand\SF{\mathscr F}

\def\mathpal#1{\mathop{\mathchoice{\text{\rm #1}}%
   {\text{\rm #1}}{\text{\rm #1}}%
   {\text{\rm #1}}}\nolimits}
\newcommand\Hess{\mathpal{Hess}}
\newcommand\Ric{\mathpal{Ric}}
\newcommand\id{\mathpal{id}}
\catcode`\@=11 
\def\mequal{\mathrel{\mathpalette\@mvereq{\hbox{\sevenrm m}}}}
\def\@mvereq#1#2{\lower.5\p@\vbox{\baselineskip\z@skip\lineskip1.5\p@
    \ialign{$\m@th#1\hfil##\hfil$\crcr#2\crcr=\crcr}}}

\def\a{\alpha}
\def\b{\beta}
\def\d{\delta}

\title[Li-Yau Type Gradient Estimates]
{Li-Yau Type Gradient Estimates and Harnack Inequalities by Stochastic Analysis}

\author[M. Arnaudon and A. Thalmaier]{Marc Arnaudon and Anton Thalmaier}

\address{
{\rm Marc Arnaudon}\\
  Laboratoire de Math\'ematiques et Applications (CNRS: UMR6086),\\
  Universit\'e de Poitiers, T\'el\'eport 2 -- BP 30179,\\ 
  F--86962 Futuroscope Chasseneuil Cedex, France}
\email{arnaudon@math.univ-poitiers.fr}
\address{
{\rm Anton Thalmaier}\\
  Unit\'e de Recherche en Math\'ematique,
  Universit\'e du Luxembourg,\\ 
  162a, avenue de la
  Fa\"{\i}encerie, L--1511 Luxembourg\\ 
  Grand-Duchy of Luxembourg}
\email{anton.thalmaier@uni.lu}

\rcvdate{}
\rvsdate{}



\begin{document}

\begin{abstract}
In this paper we use methods from Stochastic Analysis to establish 
Li-Yau type estimates for positive solutions of the heat equation. 
In particular, we want to emphasize that Stochastic Analysis provides natural tools to 
derive local estimates in the sense that the gradient bound at given point depends only 
on universal constants and the geometry of the Riemannian manifold locally about this point. 
\end{abstract}

\maketitle

\section{Introduction}
\label{Section1}
\setcounter{equation}0

The effect of curvature on the behaviour of the heat flow on a 
Riemannian manifold is a classical problem. 
Ricci curvature manifests itself most directly in gradient formulas 
for solutions of heat equation.

Gradient estimates for positive solutions of the heat equation serve as 
infinitesimal versions of Harnack inequalities: 
by integrating along curves on the manifold local gradient estimates may be 
turned into local Harnack type inequalities. 

Solutions to the heat equation
\begin{equation}\label{Eq:heat_equation0}
\frac\partial{\partial t}u=\frac12\,\Delta u
\end{equation}
on a Riemannian manifold $M$ are well understood in probabilistic terms. 
For instance, if $u=u(x,t)$ denotes the minimal solution to \eqref{Eq:heat_equation0}, 
then a straightforward calculation using It\^o's calculus leads to the 
stochastic representation of $u$ as 
\begin{equation}\label{Eq:heat_equationStoch}
u(x,t)=\E\left[1_{\{t<\zeta(x)\}}\,f(X_t(x))\right]
\end{equation}
where $f=u(\newdot,0)$ is the initial condition in \eqref{Eq:heat_equation0}, 
$X_\bull(x)$ denotes a Brownian motion on $M$, starting from $x$ 
at time~$0$, and $\zeta(x)$ its lifetime.

It is a remarkable fact that exact stochastic representation formulas for the derivative of 
solutions to the heat equation can be given, expressing the gradient $\nabla u$ of $u$
in terms of Ricci curvature. 

The following typical example for such a 
Bismut type derivative formula is taken from~\cite{Thalmaier-Wang:98}. 

\begin{thm}
[Stochastic representation of the gradient]\label{Bismut} Let 
$D$ be a relatively compact open
domain in a complete Riemannian manifold~$M$, and let $u=u(x,t)$ be a solution of the heat equation  
\eqref{Eq:heat_equation0} on $D\times{[0,T]}$ which is continuous on $\bar D\times{[0,T]}$.
Then, for any
$v\in T_{x}M$ and $x\in D$,
\begin{equation}
\langle\nabla u(\newdot,T)_{x},v\rangle=-\E\left[u(X_{\tau}(x),T-\tau)\int_{0}^{\tau}\left\langle Q_{s}
\,\dot{\ell}_{s},dB_{s}\right\rangle \right] , \label{Bismut1}
\end{equation}
where:
\begin{enumerate}
\item[\rm{(1)}] $X\equiv X_\bull(x)$ is a Brownian motion on $M$, starting at $x$, and 
$\tau=\tau(x)\wedge T$ where 
\begin{equation*}
\tau(x)=\inf\{t>0:X_{t}(x)\not \in D\}
\end{equation*}
is the first exit time from $D$; the stochastic integral is taken with respect to
the Brownian motion $B$ in $T_{x}M$, related to $X$ by the Stratonovich
equation $dB_{t}={/\!/_{t}^{-1}}\delta X_{t}$, where $\,/\!/_{t}%
^{\mstrut}\colon T_{x}M\rightarrow T_{X_{t}}M$ denotes the stochastic parallel
transport along $X$.

\item[\rm{(2)}] The process $Q$ takes values in the group of linear
automorphisms of $T_{x}M$ and is defined by the pathwise covariant ordinary
differential equation\/,
\begin{equation*}
dQ_{t}= -{\textstyle{\frac{1}{2}}}\, \Ric_{/\!/_{t}^{\mstrut}}(Q_{t}^{\mstrut})\,dt,\quad Q_{0}=\id_{T_{x}M},
\end{equation*}
where $\Ric_{/\!/_{t}^{\mstrut}}={/\!/_{t}^{-1}}\circ{\Ric_{X_{t}}^{\sharp}%
}\circ{/\!/_{t}^{\mstrut}}$ (a linear transformation of $T_{x}M$), and
$\langle\Ric_{z}^{\sharp}u,w\rangle=\Ric_{z}(u,w)$ for any $u,w\in T_{z}M$,
$z\in M$.

\item[\rm{(3)}] Finally, $\ell_{t}$ may be any adapted finite energy
process taking values in $T_{x}M$ such that $\ell_{0}=v$, $\ell_{\tau}=0$ and
\[
\bigl(\int_{0}^{\tau}|\dot{\ell}_{t}|^{2}\,dt\bigr){}^{1/2}\in
L^{1+\varepsilon}\quad\hbox{for some $\varepsilon>0$}.
\]
\end{enumerate}
\end{thm}

Formula \eqref{Bismut1} is valid for any solution $u$ of the heat equation and does not require positivity of $u$.

\begin{remark}Formula \eqref{Bismut1} is easily adapted to more specific situations, for instance:

(i) Let $u=u(x,t)$ be a solution of the heat equation  
\eqref{Eq:heat_equation0} on $D\times{[0,T]}$ such that $u|_{t=0}=f$ and $u(\newdot,t)|\partial D=f|\partial D$. 
Then
\begin{equation*}
\langle\nabla u(\newdot,T)_{x},v\rangle=-\E\left[ {f(X_{\tau}(x))}\int_{0}^{\tau}\left\langle Q_{s}
\,\dot{\ell}_{s},dB_{s}\right\rangle \right] ,\quad \tau=\tau(x)\wedge T\,.
\end{equation*}

(ii) Let $u=u(x,t)$ be a solution of the heat equation  
\eqref{Eq:heat_equation0} on $D\times{[0,T]}$ such that $u|_{t=0}=f$ and $u(\newdot,t)|\partial D=0$. Then
\begin{equation*}
\langle\nabla u(\newdot,T)_{x},v\rangle
=-\E\left[ f(X_T(x))\,1_{\{T<\tau(x)\}}\int_{0}^{\tau(x)\wedge T}\left\langle Q_{s}
\,\dot{\ell}_{s},dB_{s}\right\rangle \right] \,.
\end{equation*}
\end{remark}

Such formulas are interesting by several means. 
For instance, on a complete Riemannian manifold $M$, 
starting from the minimal solution to the heat equation
\begin{equation}\label{Eq:heat_equationStoch1}
u(x,T)=\E\left[1_{\{T<\zeta(x)\}}\,f(X_T(x))\right]
\end{equation}
with bounded initial conditions $u(\newdot,0)=f$, since for arbitrarily small 
$T>0$ Brownian motion explores the whole manifold $M$, we observe 
that the global structure of $M$ enters in formula \eqref{Eq:heat_equationStoch1}; 
lower Ricci bounds may fail and thus ``Brownian motion may travel arbitrarily fast''. 
Nevertheless, looking at the formula for the gradient $\nabla u(\newdot,T)_{x}$ and taking into account 
that $$u(X_{\tau}(x),T-\tau)=\E^{\SF_\tau}\left[1_{\{T<\zeta(x)\}}\,f(X_T(x))\right],$$
we see that Eq.~\eqref{Bismut1} reads as
\begin{equation*}
\langle\nabla u(\newdot,T)_{x},v\rangle=-\E\left[1_{\{T<\zeta(x)\}}\,f(X_T(x))\int_{0}^{\tau}\left\langle Q_{s}
\,\dot{\ell}_{s},dB_{s}\right\rangle \right]  
\end{equation*}
where $\tau=\tau(x)\wedge T$ and $\tau(x)$ the first exit time of $X_\bull(x)$ 
from an arbitrarily small chosen neighbourhood of~$x$. 
In other words, as far as the gradient at $(x,T)$ is concerned, Ricci curvature of $M$  matters 
only locally about the 
point~$x$.

No derivative of the heat equation appears in the right-hand side of Eq.~\eqref{Eq:heat_equationStoch1}, 
thus gradient formulas are well suitable for estimates, see \cite{Thalmaier-Wang:98}. 
Such inequalities easily allow to bound $\nabla u$ in terms of some uniform 
norm of $u$.

For positive solutions of the heat equation however one wants to do better: typically one seeks 
for pointwise estimates which allow (modulo additional terms if necessary) to control 
$\nabla u(\newdot,T)_{x}$ by $u(x,T)$. 
 
\begin{thm}
[Classical Li-Yau estimate \cite{Li-Yau:86}]\label{Li-Yau_Class} 
Let $M$ be complete Riemannian manifold and assume that $\Ric\geq-k$ where $k\geq0$.
Let $u$ be a strictly positive solution of  
$$\frac\partial{\partial t}u=\frac12 \Delta u\quad\text{on}\quad M\times\R_+$$ 
and let $a>1$. Then
\begin{equation}\label{Eq:Li-Yau_Class}
\left(\frac{|\nabla u|}{u}\right)^2 (x,T)-a\,\frac{\Delta u}{u}(x,T)  \leq c(n,a)
\left[k+\frac1T\right].
\end{equation}
If $\Ric\geq0$, i.e. $k=0$, then the choice $a=1$ is possible.
\end{thm}

It is a surprising fact which has been noticed by many people that 
no straightforward way to pass from Bismut type derivative formulas to Li-Yau type gradient 
estimates seems to exist. 

Li-Yau type inequalities for positive solutions $u$ of the heat equation 
aim at estimating $\nabla\log u$ rather than $\nabla u$. 
This adds an interesting non-linearity to the problem which is better to deal with in terms of 
submartingales and Bismut type inequalities than in terms of martingales 
which are the underlying concept for Bismut formulas. 
This point of view has been worked out in~\cite{ADT:2007} for local estimates 
in the elliptic case of positive harmonic functions. Such estimates in global
form, i.e., for positive harmonic functions on Riemannian manifolds, are due to
S.T.\ Yau~\cite{Yau:75}; local versions have been established by Cheng and
Yau~\cite{Cheng-Yau:75}.

In this paper we pursue the approach via submartingales to study the parabolic case.
Even if it is meanwhile quite standard to obtain Li-Yau type estimates in global form 
via analytic methods, local versions require often completely new arguments 
\cite{Bakry-Ledoux:2006, Ni:2004, Souplet-Zhang:06, Zhang:06}. 

In this paper we derive various submartingales which lead to the wanted estimates 
in a surprisingly simple way. 

\section{Basic formulas related to positive solutions of the heat equation 
and some elementary submartingales}
\label{Section2}
\setcounter{equation}0

The following formulas for solutions of the heat equation on a Riemannian manifold 
depend on the fact that the solutions are strictly positive. 

\begin{lemma}\label{Lemma:1.1}
Let $M$ be a Riemannian manifold (not necessarily complete) and let $u=u(x,t)$ be a positive 
solution of the heat equation 
\begin{equation}\label{Eq:heat_equation}
\frac\partial{\partial t}u=\frac12\,\Delta u
\end{equation}
on $M\times[0,T]$. Then the following equalities hold:
\begin{align}
\label{Eq:1.1a}
&\left(\frac12\Delta-{\partial_t}\right)\,(\log u)=-\frac12\frac{\vert\nabla u\vert^2}{u^2},\\
\label{Eq:1.1b}
&\left(\frac12\Delta-{\partial_t}\right)\,(u\log u)=\frac12\frac{\vert\nabla u\vert^2}{u},\\
\label{Eq:1.1c}
&\left(\frac12\Delta-{\partial_t}\right)\,\frac{\vert\nabla u\vert^2}{u}
=\frac1u\left|\Hess u-\frac{\nabla u\otimes\nabla u}u\right|^2+\frac{\Ric(\nabla u,\nabla u)}u.
\end{align}
\end{lemma}

\begin{proof} All three equalities are easily checked by direct calculation. 
\end{proof}

Eq.~\eqref{Eq:1.1c} in Lemma \ref{Lemma:1.1} gives raise to some inequalities 
frequently used in the sequel and crucial for our approach.
Most of our results are based on the following observation.

\begin{cor}\label{cor:1.1} 
Let $M$ be a Riemannian manifold of dimension $n$ (not necessarily complete) and let $u=u(x,t)$ be a positive 
solution of the heat equation \eqref{Eq:heat_equation}. Then we have:
\begin{equation}\label{Eq:1.4a}
\left(\frac12\Delta-{\partial_t}\right)\,\frac{\vert\nabla u\vert^2}{u}
\geq\frac1{nu}\left(\Delta u-\frac{\vert\nabla u\vert^2}{u}\right)^2+\frac{\Ric(\nabla u,\nabla u)}u.\end{equation}
If $\Ric\geq-k$ on $M$ for some $k\geq0$, then 
\begin{equation}\label{Eq:1.4b}
\left(\frac12\Delta-{\partial_t}\right)\,\frac{\vert\nabla u\vert^2}{u}
\geq\frac1{nu}\left(\Delta u-\frac{\vert\nabla u\vert^2}{u}\right)^2-k\,\frac{\vert\nabla u\vert^2}{u},\end{equation}
and in particular,
\begin{equation}\label{Eq:1.5}
\left(\frac12\Delta-{\partial_t}\right)\,\frac{\vert\nabla u\vert^2}{u}
\geq-k\,\frac{\vert\nabla u\vert^2}{u}.\end{equation}
\end{cor}

\begin{proof} This is a direct consequence of Eq.~\eqref{Eq:1.1c}.
\end{proof}

\begin{lemma}\label{Lemma:2.1}
Let $M$ be a Riemannian manifold (not necessarily complete) and let $u(x,t)=P_tf(x)$ be a positive 
solution of the heat equation 
\begin{equation}\label{Eq:heat_equation_again}
\frac\partial{\partial t}u=\frac12\,\Delta u
\end{equation}
on $M\times[0,T]$.
For any Brownian motion $X$ on~$M$, the process
\begin{align}
m^1_t:=\frac{|\nabla P_{T-t}f|^2}{P_{T-t}f}(X_t)\,\exp\left\{-\int_0^t\underline{\Ric}(X_r)\,dr\right\},
\label{Eq:2.5a}\end{align}
where $\underline{\Ric}(x)$ denotes the smallest eigenvalue of the Ricci curvature at the point $x$, 
is a local submartingale (up to its natural lifetime).
Furthermore
\begin{align}
m^2_t&:=\big(\log P_{T-t}f\big)(X_t)
+\frac12\int_0^t\left|\frac{\nabla P_{T-s}f}{P_{T-s}f}\right|^2(X_s)\,ds\\
m^3_t&:=\big(P_{T-t}f\log P_{T-t}f\big)(X_t)
-\frac12\int_0^t\frac{|\nabla P_{T-s}f|^2}{P_{T-s}f}(X_s)\,ds
\end{align}
are local martingales (up to their respective lifetimes).
\end{lemma}

\begin{proof}
The first claim is a consequence of \eqref{Eq:1.4a}; 
the second part comes from Eqs.~\eqref{Eq:1.1a} and \eqref{Eq:1.1b}.
\end{proof}

\begin{lemma}\label{Lemma:2.4}
Let $M$ be a Riemannian manifold and $u(x,t)=P_tf(x)$ be a positive 
solution of the heat equation \eqref{Eq:heat_equation_again} on $M\times{[0,T]}$.
If $\Ric\geq-k$ for some $k\geq0$, then for any Brownian motion $X$ on~$M$, the process
\begin{equation}\label{Eq:2.12}
N_t:=\frac{T-t}{2(1+k(T-t))}\,\frac{|\nabla P_{T-t}f|^2}{P_{T-t}f}(X_t)
+\big(P_{T-t}f\log P_{T-t}f\big)(X_t)
\end{equation}
is a local submartingale (up to its lifetime).
\end{lemma}

\begin{proof} The proof follows from It\^o's formula using inequality~\eqref{Eq:1.5}, 
along with Eq.~\eqref{Eq:1.1b}.
\end{proof}

\section{Global gradient estimates}
\label{Section3}
\setcounter{equation}0

In this section we explain how submartingales related to positive solutions of the heat equation 
can be turned into gradient estimates. 
The resulting estimates of this section are classical inequalities; our focus lies on the 
stochastic approach. 

The main problem in the subsequent sections will then be to use methods of Stochastic Analysis 
to derive localized versions of the bounds.
The following gradient estimates follow immediately from Lemma \ref{Lemma:2.4}.

\begin{thm}[Entropy estimate]
\label{thm:entropy}
Let $u(x,t)=P_tf(x)$ be a positive solution of the heat equation on a compact 
manifold $M$. Assume that $\Ric\geq-k$ for some $k\geq0$. Then 
\begin{equation}
\label{Eq:2.7}
\left|\frac{\nabla P_{T}f}{P_{T}f}\right|^2(x)
\leq 2\left(\frac1T+k\right)\,P_{T}\left(\frac f{P_{T}f(x)}\log \frac f{P_{T}f(x)}\right)(x).
\end{equation}
\end{thm}
 
\begin{proof}
Indeed if $M$ is compact, then the local submartingale $N_t$ in \eqref{Eq:2.12} is a true 
submartingale. Let
$$h(t)=\frac{T-t}{2(1+k(T-t))}.$$ 
Then, evaluating $\E[N_0]\leq\E[N_T]$, we obtain
\begin{equation*}
h(0)\,\frac{|\nabla P_{T}f|^2}{P_{T}f}(x)
+P_{T}f(x)\log P_{T}f(x)\leq P_{T}(f\log f)(x),
\end{equation*}
or in other words,
\begin{equation*}
\frac{|\nabla P_{T}f|^2}{P_{T}f}(x)
\leq \frac1{h(0)}\,P_{T}\left(f\log \frac f{P_{T}f(x)}\right)(x).
\end{equation*}
Dividing through $P_{T}f(x)$ completes the proof.
\end{proof}

\begin{cor}
Keeping notation and assumptions of Theorem \textup{\ref{thm:entropy}} 
we observe that, for any $\delta>0$,
\begin{align}
|\nabla P_Tf(x)|&\leq\ \frac1{2\delta}\left(\frac1T+k\right)\,P_Tf(x)\notag\\
&\quad+\delta\,\big[P_T\left(f\log f\right)(x)-P_Tf(x)\log P_Tf(x)\big].
\label{Eq:Entropy1}
\end{align}
\end{cor}

\begin{proof}
Indeed with $h(0)=T/{(2+2kT)}$, we conclude from ~\eqref{Eq:2.7} that
\begin{align*}
\frac{|\nabla P_{T}f|}{P_{T}f}(x)
&\leq \sqrt{\frac1{2\delta\,h(0)}}\,\sqrt{2\delta\,P_{T}
  \left(\frac f{P_{T}f(x)}\log\frac f{P_{T}f(x)}\right)(x)}\\
&\leq \frac1{2\delta}\,\left(\frac1T+k\right)
+\delta\,P_{T}\left(\frac f{P_{T}f(x)}\log \frac f{P_{T}f(x)}\right)(x).
\end{align*}
\end{proof}

\begin{cor}[Hamilton \cite{Hamilton:93}] Let $M$ be a compact Riemannian manifold such that 
$\Ric\geq-k$ throughout $M$ for some $k\geq0$.
Suppose that $u(x,t)$ is a positive 
solution of the heat equation \eqref{Eq:heat_equation_again} on $M\times{[0,T]}$, 
and let $A:=\sup_{M\times{[0,T]}}u$.
Then
\begin{equation}\label{Eq:Hamilton}
\frac{\vert\nabla u\vert^2}{u^2}(x,T)\leq2\left(\frac1T+k\right)\log\frac A{u(x,T)}\,.
\end{equation}
In particular, if $\Ric\geq0$ then
\begin{equation}\label{Eq:Hamilton0}
\frac{\vert\nabla u\vert}{u}(x,T)\leq\frac1{T^{1/2}}\sqrt{2\log\frac A{u(x,T)}}\,.
\end{equation}
\end{cor}

\begin{proof} 
The proof of \eqref{Eq:Hamilton} is an application of Theorem \ref{thm:entropy}, along 
with the observation that ${P_Tf^*}(x)=1$ when $f$ is normalized as $f^*:=f/{P_Tf(x)}$ for fixed $x$.
\end{proof}

\begin{remark}
In the proofs above compactness of the manifold has only been used 
to assure that the local submartingale \eqref{Eq:2.12} is a true submartingale.
The results extend to bounded positive solution of the heat equation 
on complete manifolds with lower Ricci bounds. 
Indeed, $u(x,t)=P_tf(x)$ may be assumed to be bounded away from $0$ (otherwise 
one may first pass to $u^*:=u+\varepsilon$ and let $\varepsilon>0$ tend to~$0$ 
in the obtained estimate). The term $(T-t)\,\vert\nabla P_{T-t}f\vert^2$ may be bounded 
by Bismut's formula, see \cite{Thalmaier-Wang:98}. 
\end{remark}

We now turn to the classical Li-Yau estimate \eqref{Eq:Li-Yau_Class}.
Let $M$ be a given complete Riemannian manifold of dimension $n$ 
such that $\Ric\geq-k$ for some $k\geq0$.
Suppose that $u$ is a positive solution to the heat equation \eqref{Eq:heat_equation_again} 
on $M\times{[0,T]}$.
Starting from \eqref{Eq:1.4b} in Corollary \ref{cor:1.1} we have 
\begin{equation}\label{Eq:1.4q}
\left(\frac12\Delta-{\partial_t}\right)\frac{\vert\nabla u\vert^2}{u}
\geq\frac1{nu}\left(\Delta u-\frac{\vert\nabla u\vert^2}{u}\right)^2-k\frac{\vert\nabla u\vert^2}{u},
\end{equation}
 and thus
\begin{equation}\label{Eq:1.4q_again}
\left(\frac12\Delta-{\partial_t}\right)\frac{\vert\nabla u\vert^2}{u}
\geq\frac1{nu}\,q^2-k\,q-k\,\Delta u,
\end{equation}
where 
\begin{equation*}
q:=\frac{\vert\nabla u\vert^2}{u}-\Delta u=\frac{\vert\nabla u\vert^2}{u}-2\,{\partial_t}u.
\end{equation*}
Fixing $x\in M$, let
$X=X(x)$ be a Brownian motion on $M$ starting at~$x$. 
Our first goal is to investigate the process
$$\left(h_t\,\frac{\vert\nabla u\vert^2}{u}(X_t,T-t)\right)_{t\geq0}$$
where $h_t=\ell_t^2$ for some adapted continuous real-valued process $\ell_t$ 
with absolutely continuous paths such that $\ell_0=1$ and $\ell_T=0$.

Let $q_t=q(X_t,T-t)$, $u_t=u(X_t,T-t)$ and $(\Delta u)_t=\Delta u(X_t,T-t)$.
Using \eqref{Eq:1.4q_again} we find (modulo differentials of local martingales) 
\begin{align}\label{Eq:ht_qt}
d&\left(h_t\,\frac{\vert\nabla u\vert^2}{u}(X_t,T-t)\right)\\
&\qquad\geq \left[\dot h_tq_t+\frac{h_t}{nu_t}\,q_t^2-k\,q_th_t+\left(\dot h_t-k\,h_t\right)(\Delta u)_t\right]dt.
\end{align}
Minimizing the term 
$$\frac{h_t}{nu_t}\,q_t^2+(\dot h_t-k\,h_t)q_t$$ 
as a quadratic function of $q_t$, we find 
$$\frac{h_t}{nu_t}\,q_t^2+(\dot h_t-k\,h_t)q_t\geq -\frac{(\dot h_t-k\,h_t)^2}{4h_t}\,nu_t.$$
Thus, integrating \eqref{Eq:ht_qt} from $0$ to $T$ and taking expectations, we obtain 
\begin{equation}\label{Eq:ht_qt1}
\frac{\vert\nabla u\vert^2}{u}(x,T)
\leq\E\left[\int_0^T \left(n\,\frac{(\dot h_t-k\,h_t)^2}{4h_t}\,u_t+(kh_t-\dot h_t)(\Delta u)_t\right) dt
\right]
\end{equation}

\begin{thm}[Li-Yau inequality for $\Ric\geq0$]\label{thm:Li_Yau_Ric0} 
Let $M$ be a complete Riemannian manifold of dimension $n$ 
such that $\Ric\geq0$. 
Let $u=u(x,t)$ be a positive bounded solution of the heat equation \eqref{Eq:heat_equation_again} 
on $M\times{[0,T]}$.
Then, for each $x\in D$, 
\begin{equation}\label{Eq:Li_Yau_Ric0}
\frac{\vert\nabla u\vert^2}{u^2}(x,T)-\frac{\Delta u}{u}(x,T)
\leq n\, \E\left[\int_0^T |\dot\ell_t|^2 \,\frac{u(X_t(x),T-t)}{u(x,T)}\,dt
\right]
\end{equation}
where $(\ell_t)$ is an adapted continuous real-valued process $\ell_t$ 
with absolutely continuous paths such that $\ell_0=1$ and $\ell_T=0$.

In particular, with the choice $\ell_t:=(T-t)/T$, we obtain 
\begin{equation*}
\frac{\vert\nabla u\vert^2}{u^2}(x,T)-\frac{\Delta u}{u}(x,T)
\leq \frac nT\,.
\end{equation*}
which is the classical estimate of Li-Yau. 
\end{thm}

\begin{proof}
By \eqref{Eq:ht_qt} we have 
\begin{equation}\label{Eq:ht_qt1a}
\frac{\vert\nabla u\vert^2}{u}(x,T)
\leq\E\left[\int_0^T\left(n\,\frac{|\dot h_t|^2}{4h_t}\,u_t-\dot h_t\,(\Delta u)_t\right) dt.
\right]
\end{equation}
First note that $\Delta u$ is a solution of the heat equation as well, hence $(\Delta u)_t$ a 
local martingale. In particular,
$$
d\big(h_t\,(\Delta u)_t\big)\mequal\dot h_t\,(\Delta u)_t\,dt,
$$
where $\mequal$ stands for equality modulo differentials of local martingales, and hence
$$\E\left[\int_0^T \dot h_t\,(\Delta u)_t\, dt\right]=-h_0\,(\Delta u)_0=-\Delta u(x,T).
$$
Thus \eqref{Eq:ht_qt1a} shows that
\begin{equation*}
\frac{\vert\nabla u\vert^2}{u^2}(x,T)-\frac{\Delta u}{u}(x,T)
\leq n\,\E\left[\int_0^T \dot\ell_t^2\,\frac{u_t}{u_0}\,dt
\right]=n\,\E\left[\int_0^T\dot\ell_t^2\,dt\cdot \frac{u_T}{u_0}\right].
\end{equation*}
\end{proof}

\begin{remarks}
(i) \  Using \eqref{Eq:ht_qt1}, Theorem \ref{thm:Li_Yau_Ric0} 
is easily extended to the case of a lower Ricci bound. 
For local versions of the Li-Yau's estimate one could try to modify the process $\ell_t$
in such a way that $\ell_t$ already vanishes as soon as the Brownian motion $X_t$ 
reaches the boundary of~$D$. We shall not pursue this approach here, but rather adopt 
an even simpler argument in the next section which leads to the local estimate.

(ii) \ Of particular interest are localized versions of the entropy estimates \eqref{Eq:2.7} 
and \eqref{Eq:Entropy1}. Such estimates lead to Harnack inequalities and 
heat kernel bounds, valid on arbitrary manifolds without bounded geometry, see
\cite{Arnaudon-Thalmaier-Wang:06, Arnaudon-Thalmaier-Wang:08}.
Results in this direction will be worked out elsewhere~\cite{Arnaudon-Thalmaier-Wang:09}. 

\end{remarks}

\section{Local Li-Yau type inequalities}
\label{Section4}
\setcounter{equation}0

Our main task of the remaining sections will be to localize the arguments of 
Section \ref{Section3} to cover local solutions of the heat equation on bounded 
domains. We start with the Li-Yau estimate.

\begin{assumption}
\label{Ass1}
Let $M$ be a complete Riemannian manifold of dimension $n$, and  
let $u=u(x,t)$ be a solution of the heat equation \eqref{Eq:heat_equation_again} 
on $D\times{[0,T]}$ where $D\subset M$ is a relatively compact open subset of $M$ 
with nonempty smooth boundary. 
Assume that $u$ is positive and continuous on $\bar D\times{[0,T]}$. 
Furthermore let
\begin{equation}
\label{Eq:RicciBound}
k:=\inf\bigl\{\Ric_x(v,v): v\in T_xM,\ |v|=1,\ x\in D\bigr\}
\end{equation}
be a lower bound for the Ricci curvature on the domain $D$. 
Finally let $\varphi\in C^2(\bar D)$ with $\varphi>0$ in $D$ and $\varphi|\partial D=0$.
\end{assumption}

\begin{assumption}
\label{Ass2}
Let $x\in D$ and $X(x)$ be a Brownian motion on $M$ starting at~$x$ at time~$0$. 
Denote by $\tau(x)$ its first exit time of~$D$. 
\end{assumption}

Now, for $t\in [0,T\wedge\tau(x)[$, consider the process 
\begin{equation}
 \label{E1}
Y_t=C_1\,(T-t)^{-1}+C_2\,\varphi^{-2}(X_t(x))+\alpha k,
\end{equation}
where $C_1$, $C_2>0$ are constants which will be specified later.\goodbreak 

Let $h_t$ be the solution of 
\begin{equation}
 \label{E2}
\dot h_t=-h_tY_t,\qquad h_0=1;
\end{equation}
in other words,
\begin{equation*}
h_t=\exp\left\{-\int_0^t\left(\frac{C_1}{T-r}+\frac{C_2}{\varphi^{2}(X_r(x))}+\alpha k\right)dt\right\}.
\end{equation*}
Then, letting
\begin{equation*}
q_t=\left(\frac{|\nabla u|^2}{u}-\Delta u\right)(X_t,T-t),
\end{equation*}
we find (modulo differentials of local martingales) 
\begin{equation}
 \label{E3}
d(h_tq_t)\ge \left[h_t\frac{q_t^2}{nu_t}+\left(\dot h_t-kh_t\right)q_t-kh_t(\Delta u)_t\right] dt.
\end{equation}
Thus letting 
\begin{equation}
 \label{E4}
S_t:=h_tq_t+nu_t\dot h_t\equiv h_t\,(q_t-nu_tY_t),\quad t\in [0,T\wedge\tau(x)[,
\end{equation}
we get (modulo differentials of local martingales)
\begin{align*}
 dS_t&=d(h_tq_t)-nu_t\,d(h_tY_t)+n\,d[u,\dot h]_t\\
&\ge \left[h_t\frac{q_t^2}{nu_t}+\left(\dot h_t-kh_t\right)q_t-kh_t(\Delta u)_t\right]dt\\
&\quad+nu_t\,h_tY_t^2\,dt\\
&\quad-nu_th_t\left[C_1\,(T-t)^{-2}+C_2\,c_\varphi(X_t)\,\varphi^{-4}(X_t)\right]dt\\
&\quad-nC_2\,h_t\,d[ u,\varphi^{-2}(X(x))]_t
\end{align*}
where the bracket $[\newdot,\newdot]$ stands for quadratic covariation on the space of 
continuous semimartingales and where
\begin{equation}
 \label{E5}
c_\varphi(x)=\left(3|\nabla \varphi|^2-\varphi\Delta \varphi\right)(x).
\end{equation}
On the other hand, denoting again 
\begin{align}
\begin{split}
\label{Eq:Defu_t}
u_t&=u(T-t,\cdot)(X_t),\\ \nabla u_t&=\nabla u(T-t,\cdot)(X_t) 
\end{split}
\intertext{and}
\Delta u_t&=\Delta u(T-t,\cdot)(X_t),
\label{Eq:DefDeltau_t}
\end{align}
we have for any $\alpha>0$,
\begin{align*}
-n&C_2\,h_t\,d[ u,\varphi^{-2}(X_t)]_t\\
&=nC_2\,h_t2\varphi^{-3}(X_t)\langle \nabla u_t,\nabla \varphi(X_t)\rangle\,dt\\
&\ge -2nC_2\,h_t\varphi^{-3}(X_t)|\nabla u_t|\,|\nabla \varphi(X_t)|\,dt\\
&=-2nC_2\,h_t\left(\varphi^{-1}(X_t)\,(\alpha nu_t)^{-1/2}\,|\nabla u_t|\right)\\
&\qquad\times\left((\alpha nu_t)^{1/2}(\varphi^{-2}|\nabla \varphi|)(X_t)\right)\,dt\\
&\ge \left(-\alpha^{-1}C_2\,h_t\varphi^{-2}(X_t)(q_t+\Delta u_t)
-C_2\,\alpha n^2u_th_t\varphi^{-4}|\nabla \varphi|^2(X_t)\right)\,dt\\
&\ge \left(-\alpha^{-1}h_t(Y_t-\alpha k)(q_t+\Delta u_t)
-C_2\,\alpha n^2u_th_t\varphi^{-4}|\nabla \varphi|^2(X_t)\right)\,dt\\
&= \left(\left(\alpha^{-1}\dot h_t+kh_t\right)(q_t+\Delta u_t)
-C_2\,\alpha n^2u_th_t\varphi^{-4}|\nabla \varphi|^2(X_t)\right)\,dt.
\end{align*}
Hence letting, for $\a>0$,
\begin{equation}
 \label{E6}
S_{\a,t}=h_t\left(q_t-\a^{-1}\Delta u_t\right)+nu_t\dot h_t
\end{equation}
we get 
\begin{align*}
 dS_{\a,t}&\ge \left[h_t\,\frac{q_t^2}{nu_t}+\left(1+\a^{-1}\right)\dot h_t\,q_t\right]dt\\
&\quad+h_tY_t^2\,nu_t\,dt\\
&\quad-nu_th_t\Big[C_1(T-t)^{-2}+C_2\,C_{\varphi,\a,n}\,\varphi^{-4}(X_t)\Big]\,dt
\end{align*}
where 
\begin{equation}
 \label{E7}
C_{\varphi,\a,n}=\sup_D\left\{c_\varphi+\a n|\nabla \varphi|^2\right\}
=\sup_D\left\{(3+\a n)|\nabla \varphi|^2-\varphi\Delta \varphi\right\}.
\end{equation}
Minimizing the first line on the right hand side, we find
\begin{align*}
 dS_{\alpha,t}&\ge h_t\,Y_t^2\,nu_t\left(\frac{4-(1+\alpha^{-1})^2}{4}\right)\,dt\\
&\quad-nu_t\,h_t\Big[C_1\,(T-t)^{-2}+C_2\,C_{\varphi,\alpha,n}\,\varphi^{-4}(X_t)\Big]\,dt.
\end{align*}
Putting things together, we arrive at the following result.

\begin{lemma}
We keep notation and assumptions from above. Assume that $\alpha>1$ 
and let
\begin{equation}
 \label{E9}
C_1\ge C_1(\a):=\frac4{4-(1+\a^{-1})^2} 
\end{equation}
and
\begin{align}
 \label{E10}
\begin{split}
C_2&\ge C_2(\varphi,\a,n)\\
&:= \frac4{4-(1+\a^{-1})^2}\,\sup_D\Big\{(3+\a n)|\nabla \varphi|^2-\varphi\Delta \varphi\Big\}.
\end{split}
\end{align}
Then $S_{\a,t}$ is a local submartingale on $[0,T\wedge\tau(x)[$.
\end{lemma}

\begin{proof}
Using 
\begin{equation}
 \label{E8}
Y_t^2\ge C_1^2\,(T-t)^{-2}+C_2^2\,\varphi^{-4}(X_t),
\end{equation}
we get 
\begin{align*}
 dS_{\a,t}&\ge\ h_tnu_t\,C_1\left(\frac{4-(1+\a^{-1})^2}{4}\,C_1-1\right)(T-t)^{-2}\,dt\\
&\ +h_tnu_t\,C_2\left(\frac{4-(1+\a^{-1})^2}{4}\,C_2-C_{\varphi,\a,n}\right)\varphi^{-4}(X_t)\,dt.
\end{align*}
Thus, under condition \eqref{E9} and \eqref{E10}, the right-hand side is nonnegative.
\end{proof}

\begin{thm}[Li-Yau inequality; local version]
\label{T1}
We keep the assumptions from above. Let $u=u(x,t)$ be a solution of the heat equation 
on $D\times{[0,T]}$ which is positive and continuous on $\bar D\times{[0,T]}$. 

For any $\a\in{]1,\infty[}$,
we have 
\begin{equation}
\label{E11}
\frac{|\nabla u_0|^2}{u_0^2}-(1+\a^{-1})\,\frac{\Delta u_0}{u_0}
\le \frac{n\,C_1(\a)}{T}+\frac{n\,C_2(\varphi,\a,n)}{\varphi^2(x)}+n\a k,
\end{equation}
where $k$ is a lower Ricci bound on the domain $D$, and where $C_1(\a)$ and $C_2(\varphi,\a,n)$ 
are specified in \eqref{E9}, resp.~\eqref{E10}. 
The function $\varphi$ is 
as in Assumption \textup{\ref{Ass1}}.

Recall that $u_0=u(\newdot,T)$ and $\Delta u_0=\Delta u(\newdot,T)$ according to 
\eqref{Eq:Defu_t}, resp. \eqref{Eq:DefDeltau_t}
\end{thm}

\begin{proof}
Let $$C_1=C_1(\a)\quad\text{and}\quad C_2=C_2(\varphi,\a,n).$$
Consider 
$$
S_{\a,t}=h_t\left(q_t-\a^{-1}\Delta u_t-nu_tY_t\right).
$$
We assume that $S_{\a,0}>0$ and let $\tau_\a$ be the first hitting time of $0$ by~$S_{\a,t}$. 
Then clearly $\tau_\a<T\wedge\tau(x)$, since $$q_t-\a^{-1}\Delta u_t-nu_tY_t$$ 
converges to $-\infty$ as $t$ tends to $T\wedge\tau(x)$.

Let 
\begin{equation*}
K:=\max_{ \bar D\times [0,T]}(q-\alpha^{-1}\Delta u)\quad\text{and}
\quad m=\min_{\bar D\times [0,T]}u.
\end{equation*}
Then, for $t\in [0,\tau_\a]$, we have $S_t\ge 0$, which implies $Y_t\le {K}/nm$.
From this we easily prove that on $[0,\tau_\a]$, 
the process $S_{\alpha,t}$ is a submartingale 
with bounded local characteristics. 
As a consequence, we have 
$$
S_{\a,0}\le \E[S_{\a,\tau_\a}].
$$
Since $S_{\a,\tau_\a}\equiv 0$, this contradicts the assumption $S_{\a,0}>0$. 
Hence we must have $S_{\a,0}\le 0$, which is the desired inequality.\end{proof}

\begin{remark}
 In the case of a global solution of the heat equation on a compact manifold, 
we can take $\varphi$ as a constant and 
then $C_2(\alpha, \varphi, n)=0$. If moreover $k=0$, then one can take 
$C_2=0$ and $C_1=1$, and $S_t$ is a local submartingale. 
This recovers one more time the usual Li-Yau estimate
\begin{equation}
 \label{E12}
\frac{|\nabla u_0|^2}{u_0^2}-\frac{\Delta u_0}{u_0}\le \frac{n}{T}\,.
\end{equation}
A similar reasoning applies for global solutions on complete Riemannian manifolds 
with a lower Ricci curvature bound. 
\end{remark}

\section{Li-Yau inequality with lower order term}
\label{Section5}
\setcounter{equation}0

We keep Assumption \ref{Ass1} and \ref{Ass2} of Section \ref{Section4} as standing assumptions 
for the rest of the paper and 
study now the process
\begin{equation}
 \label{E13}
S_t'=q_t-nu_tZ_t
\end{equation}
where
$$Z_t=C_1\,(T-t)^{-1}+C_2\,\varphi^{-2}(X_t)+C_3$$
with constants $C_1, C_2, C_3>0$ to be specified later. 

Let
\begin{equation}
\label{E18}
C_\varphi=\sup_D\left\{3|\nabla \varphi|^2-\varphi\Delta \varphi\right\}.
\end{equation}
Then we have
\begin{align*}
dS_t'&\ge \left[\frac{q_t^2}{nu_t}-k\frac{|\nabla u_t|^2}{u_t}
-nu_t\left(C_1(T-t)^{-2}+C_2\,C_\varphi\,\varphi^{-4}(X_t)\right)\right]dt\\
&\qquad-C_2\,nd\,[\varphi^{-2}(X),u]_t\\
&\ge \left[\frac{q_t^2}{nu_t}-nu_t\left(\frac{C_1}{(T-t)^2}
+\frac{C_2\,C_\varphi}{\varphi^4(X_t)}
+\frac{k}{n}\left\|\frac{|\nabla u|}{u}\right\|_{\bar D\times [0,T]}^2\right)\right]dt\\
&\qquad-2C_2\,nu_t\left\|\frac{|\nabla u|}{u}\right\|_{\bar D\times [0,T]} 
{\left\| \varphi\nabla \varphi\right\|_D^2}\,{\varphi^{-4}(X_t)}\, dt\\
&\ge \left[\frac{q_t^2}{nu_t}-nu_t\left(\frac{C_1}{(T-t)^2}+\frac{C_2}{\varphi^{4}(X_t)}
\phantom{\sqrt\frac12}\right.\right.\\
&\qquad
\times\left.\left.\left(C_\varphi+2\left\|\frac{|\nabla u|}{u}\right\|_{\bar D\times [0,T]}\left\| 
\varphi\nabla \varphi\right\|_D\right)\right)\right]dt\\
&\qquad-nu_t\,\frac{k}{n}\left\|\frac{|\nabla u|}{u}\right\|_{\bar D\times [0,T]}^2dt.
\end{align*}

\begin{lemma}
Let 
\begin{align}
 \label{E14}\begin{split}
&C_1=1,\quad C_2
=C_\varphi+2\left\|\frac{|\nabla u|}{u}\right\|_{\bar D\times [0,T]}\left\| \varphi\nabla 
\varphi\right\|_\infty\quad\text{and}\\
&C_3=\sqrt{\frac{k}{n}}\left\|\frac{|\nabla u|}{u}\right\|_{\bar D\times [0,T]}.
\end{split}
\end{align}
Then on $\{S_t'\ge 0\}$, the process $S_t'$ has nonnegative drift.
\end{lemma}

\begin{proof}
 We have
$$
Z_t^2\ge C_1^2(T-t)^{-2}+C_2^2\varphi^{-4}(X_t)+C_3^2.
$$
Consequently, under condition \eqref{E14}, 
\begin{align*}
 dS_t'\ge \frac{1}{nu_t}\left(q_t^2-(nu_tZ_t)^2\right)\,dt
=\frac{S_t'}{nu_t}\left(q_t+nu_tZ_t\right)\,dt
\end{align*}
and the right-hand side is nonnegative on $\{S_t'\ge 0\}$.
\end{proof}

Similarly to Theorem~\ref{T1}, we obtain the following result.\goodbreak

\begin{thm}[Local Li-Yau inequality with lower order term]
 \label{T2}
We keep the notation from above, as well as Assumption \textup{\ref{Ass1}} and \textup{\ref{Ass2}} 
from Section \textup{\ref{Section4}}. 
Let $u=u(x,t)$ be a solution of the heat equation 
on $D\times{[0,T]}$ which is positive and continuous on $\bar D\times{[0,T]}$. 
Then 
\begin{equation*}
\frac{|\nabla u_0|^2}{u_0^2}-\frac{\Delta u_0}{u_0}\le \frac{n}{T}
+\frac{nC_\varphi}{\varphi^2(x)}+\left(\sqrt{nk}
+\frac{2n\left\| \varphi\nabla \varphi\right\|_D}{\varphi^2(x)}\right)
\left\|\frac{|\nabla u|}{u}\right\|_{\bar D\times [0,T]}.
\end{equation*}
\end{thm}

\section{Local gradient estimates of Hamilton type}
\label{Section6}
\setcounter{equation}0

We keep Assumptions \ref{Ass1} and \ref{Ass2} 
of Section \textup{\ref{Section4}} and study now the process
\begin{equation}
\label{E15}
S_t=\frac{|\nabla u_t|}{2u_t}-u_t(1-\log u_t)^2 Z_t,
\end{equation}
where $$Z_t=\frac{C_1}{(T-t)}+\frac{C_2}{\varphi^{2}(X_t)}+C_3$$
for some constants $C_1$, $C_2$, $C_3>0$. 

Assume that $0<u\le e^{-3}$ (this assumption will be removed in Theorem~\ref{T3} 
through replacing $u$ by $e^{-3}\,u/\|u\|_{\bar D\times [0,T]}$).
Let $c_\varphi(x)$ again be given by \eqref{E5}. 
Then, denoting 
$$ g(t,x)=\frac{|\nabla u|^2}{u}(t,x)\quad\text{and}\quad g_t=g(T-t,X_t),$$
and using the fact that 
\begin{align*}
&\nabla\left(u(1-\log u)^2\right)=(\log^2 u-1)\nabla u 
\intertext{and} 
&d\left(u_t(1-\log u_t)^2\right)\mequal g_t\log u_t\,dt,
\end{align*}
we get
\begin{align*}
dS_t&\ge -k\, \frac12\,g_t\, dt\\
&\quad -u_t(1-\log u_t)^2
\Big(C_1\,(T-t)^2+C_2\,c_\varphi(X_t)\varphi^{-4}(X_t)\Big)\,dt\\
&\quad-2\log u_t\, \frac12\,g_tZ_t\, dt\\
&\quad-2C_2(1-\log u_t)(1+\log u_t)\varphi^{-3}(X_t)\,\langle \nabla u_t,\nabla \varphi(X_t)\rangle\,dt.
\end{align*}
Now from $u\le e^{-3}$ we get $$-2\log u_t\ge 3(1-\log u_t)/2.$$ 
This together with $|1+\log u_t|\le 1-\log u_t$ yields (modulo differentials of local martingales)
\begin{align*}
dS_t&\ge (1-\log u_t)\Bigg\{\left[\left(\frac32 Z_t-k\right)
 \frac12\,g_t\vphantom{\frac12}\right.\\
&\quad\left.\vphantom{\frac12}-u_t(1-\log u_t)\left(C_1(T-t)^2
+C_2c_\varphi(X_t)\varphi^{-4}(X_t)\right)\right] dt\\
&\quad-C_2\left[(1-\log u_t)\,2\sqrt 2\, \varphi^{-2}(X_t)\,|\nabla 
\varphi(X_t)|\,\sqrt{u_t}\, \varphi^{-1}(X_t)\,\sqrt{\frac{g_t}2\,}\,\right]dt\Bigg\}\\
&\ge (1-\log u_t)\Bigg\{\left[\left(\frac32 Z_t-k\right)\frac12\,g_t\vphantom{\frac12}\right.\\
&\quad\left.\vphantom{\frac12}-u_t(1-\log u_t)
 \left(C_1(T-t)^2+C_2c_\varphi(X_t)\varphi^{-4}(X_t)\right)\right] dt\\
&\quad-\left[(1-\log u_t)^2u_t\,4C_2\varphi^{-4}(X_t)\,|\nabla \varphi(X_t)|^2
+\frac14 g_tC_2\varphi^{-2}(X_t)\right] dt\Bigg\}\\
&\ge (1-\log u_t)\biggl[\left( Z_t-k\right)\frac12\,g_t-u_t(1-\log u_t)^2\\
&\quad \times\Big(C_1(T-t)^2+C_2\left(c_\varphi(X_t)+4|\nabla \varphi(X_t)|^2\right)\varphi^{-4}(X_t)\Big)\bigg]dt.
\end{align*}
Letting 
\begin{equation}
\label{E45}
C_1=1,\quad C_2=\sup_D\left\{c_\varphi+4|\nabla \varphi|^2\right\}\quad\hbox{and}\quad C_3=k,
\end{equation}
 we get 
\begin{align*}
dS_t&\ge (1-\log u_t)\left[
(Z_t-k)\frac12\,g_t-u_t(1-\log u_t)^2(Z_t-k)^2
\right]\,dt\\
&\ge (1-\log u_t)(Z_t-k)\left[
\frac12\,g_t-u_t(1-\log u_t)^2Z_t
\right]\,dt.
\end{align*}
This proves that $S_t$ has nonnegative drift on $\{S_t\ge 0\}$. 
On the other hand, $S_t$ converges to $-\infty$ as $t\to T\wedge \tau(x)$. 

Similarly to Theorem~\ref{T1}, we obtain the following result.

\begin{thm}[Local Li-Yau inequality of Hamilton type]
\label{T3} 
We keep the assumptions as above.
Assume that $u$ is a solution of the heat equation on $D\times{[0,T]}$ which 
is positive and continuous on $\bar D\times{[0,T]}$. Then 
\begin{equation*}
\left|\frac{\nabla u_0}{u_0}\right|^2
\le 2\left(\frac1T+\frac{\sup_D\left\{7|\nabla\varphi|^2-\varphi\Delta \varphi\right\}}{\varphi^2(x)}+k\right)
\left(4+\log \frac{\|u\|_{\bar D\times [0,T]}}{u_0}\right)^2\!
\end{equation*}
where $\varphi$ is as above.
\end{thm}

\section{Explicit upper bounds}
\label{Section7}
\setcounter{equation}0

The estimates in Theorems \ref{T1}, \ref{T2} and \ref{T3} have been given in terms of 
a function $\varphi\in C^2(\bar D)$ such that $\varphi>0$ in $D$ and $\varphi|\partial D=0$. 
To specify the constants an explicit choice for $\varphi$ has to be done.

We fix $x\in D$ and let $\d_x=\rho(x,\partial D)$ where $\rho$ denotes the Riemannian distance. 
We replace $D$ by the ball $B=B(x,\d_x)$ and 
consider on $B$
\begin{equation}
 \label{E20}
\varphi(y)=\cos\frac{\pi \rho(x,y)}{2\d_x}.
\end{equation}
Clearly $\varphi(x)=1$, $\varphi$ is nonnegative and bounded by $1$, and $\varphi$ vanishes on $\partial B$.

It is proven in~\cite{Thalmaier-Wang:98} that 
\begin{equation}
 \label{E19}
d\varphi^{-2}(X_t)\le \frac12 \Delta (\varphi^{-2})(X_t)\,dt 
\end{equation}
where by convention $\Delta \varphi^{-2}=0$ at points where $\varphi^{-2}$ is not differentiable. 
Moreover, since the time spent by $X_t$ on the cut-locus of $x$ is a.s. zero, the differential of 
the brackets $[\varphi(X_t),u_t]$ may be taken as~$0$ at points where $\varphi^{-2}$ 
is not differentiable. As a consequence, all estimates in 
Theorems \ref{T1}, \ref{T2} and \ref{T3} remain valid with $\varphi$ defined by \eqref{E20}. 

We are now going to derive explicit expressions for the constants. 
To this end we observe that
\begin{equation*}
\|\nabla \varphi\|_B\le \frac{\pi}{2\d_x}.
\end{equation*}
From~\cite{Thalmaier-Wang:98}, we get 
\begin{equation*}
-\Delta \varphi\le \frac{\pi\sqrt{k(n-1)}}{2\d_x}+\frac{\pi^2 n}{4\d_x^2}
\end{equation*}
which gives for any $\b>0$,
\begin{equation*}
-\Delta \varphi\le \frac{\pi^2(1+\b)n}{4\d_x^2}+\frac{k}{4\b}.
\end{equation*}
This yields 
\begin{equation*}
C_\varphi\le \frac{\pi^2\left[(1+\b)n+3\right]}{4\d_x^2}+\frac{k}{4\b},
\end{equation*}
\begin{equation*}
C_{\varphi,\a,n}\le \frac{\pi^2\left[(1+\b+\a)n+3\right]}{4\d_x^2}+\frac{k}{4\b},
\end{equation*}
\begin{equation*}
\sup_D\left\{7|\nabla \varphi|^2-\varphi\Delta \varphi\right\}\le 
\frac{\pi^2\left[(1+\b)n+7\right]}{4\d_x^2}+\frac{k}{4\b}.
\end{equation*}
Finally we replace $\a$ by $a=1+\a^{-1}$ to obtain from Theorems \ref{T1}, \ref{T2} and \ref{T3} 
the following explicit upper bounds.

\begin{thm}[Local Li-Yau inequalities with explicit constants]
 \label{C1}
Let $u$ be a solution of the heat equation 
on $D\times{[0,T]}$ where $D$ is a relatively compact open subset of a Riemannian manifold $M$. 
Assume that $u$ is positive and continuous on $\bar D\times{[0,T]}$. 
Furthermore let $k$ be a lower bound for the Ricci curvature on the domain $D$. 

Fix $x\in D$ and let $a\in{]1,2[}$. 
For any $\b>0$ we have 
\begin{align*}
\left|\frac{\nabla u_0}{u_0}\right|^2-a\frac{\Delta u_0}{u_0}
&\le \frac{4n}{(4-a^2)T}+\frac{\pi^2n\left[\left(1+\b+\frac{1}{a-1}\right)n+3\right]}{(4-a^2)\d_x^2}\\
&\quad+\left(\frac{1}{(4-a^2)\b}+\frac{1}{a-1}\right) nk,\\
\left|\frac{\nabla u_0}{u_0}\right|^2-\frac{\Delta u_0}{u_0}&\le \frac{n}{T}
+\frac{n\pi^2\left[(1+\b)n+3\right]}{4\d_x^2}+\frac{nk}{4\b}\\
&\quad+\sqrt{nk}\left\|\frac{|\nabla u|}{u}\right\|_{\bar D\times [0,T]}
+\frac{n\pi}{\d_x}\left\|\frac{|\nabla u|}{u}\right\|_{\bar D\times [0,T]}
\intertext{and}
\left|\frac{\nabla u_0}{u_0}\right|^2
&\le 2\left(\frac1T+\frac{\pi^2\left[(1+\b)n+7\right]}{4\d_x^2}+\left(\frac{1}{4\b}+1\right)k\right)\\
&\qquad\qquad\qquad\qquad\times\left(4+\log \frac{\|u\|_{\bar D\times [0,T]}}{u_0}\right)^2
\end{align*}
where $n$ denotes the dimension of $M$ and $\d_x$ the Riemannian distance of $x$ to the boundary of $D$.
Recall that $u_0=u(\newdot,T)$ and $\Delta u_0=\Delta u(\newdot,T)$.

\end{thm}\goodbreak


\end{document}